\documentclass[a4paper,12pt]{article}
\usepackage{amsmath,amssymb,amsthm,verbatim,mathrsfs}
\newcommand{\C}{{\mathbb C}}
\newcommand{\N}{{\mathbb N}}
\newcommand{\R}{{\mathbb R}}
\newcommand{\Z}{{\mathbb Z}}

\newcommand{\abs}[2][\empty]{\ifx#1\empty\left|#2\right|%
\else#1\vert #2 #1\vert\fi}
\newcommand{\caninf}{\alpha}
\newcommand{\Cnt}[1][]{{\cal C}^{#1}}
\newcommand{\conv}{\star}
\newcommand{\csubset}{\subset\subset}
\newcommand{\defstyle}[1]{{\em #1}}
\newcommand{\eps}{\varepsilon}

\newcommand{\FinPart}[2][\empty]{\ifx#1\empty\mathrm{Pf}\left(#2\right)%
\else\mathrm{Pf}#1(#2#1)\fi}

\newcommand{\integr}{\int\limits}

\newcommand{\lspan}{\mathop\mathrm{span}}
\newcommand{\meas}{\mu}

\newcommand{\norm}[2][\empty]{\ifx#1\empty\left\Vert#2\right\Vert%
\else#1\Vert #2 #1\Vert\fi}

\newcommand{\pairing}[3][\empty]{\ifx#1\empty\left\langle#2,#3\right\rangle%
\else#1\langle #2,#3 #1\rangle\fi}
\newcommand{\restr}[2]{{#1}_{|#2}}

\newcommand{\supp}{\mathop\mathrm{supp}}
\newcommand{\Schwartz}{\mathscr{S}}
\newcommand{\test}{\mathcal D}

\newcommand{\Gen}{{\mathcal G}}
\newcommand{\GenC}{\widetilde\C}
\newcommand{\GenR}{\widetilde\R}
\newcommand{\Mod}{{\mathcal M}}
\newcommand{\Null}{{\mathcal N}}

\newtheorem*{df}{Definition}

\newtheorem{thm}{Theorem}[section]
\newtheorem{prop}[thm]{Proposition}
\newtheorem{lemma}[thm]{Lemma}
\newtheorem{cor}[thm]{Corollary}
\newtheorem{example}[thm]{Example}

\theoremstyle{remark}
\newtheorem*{rem}{Remark}

\begin{document}
\title{Weak homogeneity in generalized function algebras}
\author{Hans Vernaeve\footnote{Supported by FWF(Austria), grants M949, Y237}\\Unit for Engineering Mathematics\\University of Innsbruck, A-6020 Innsbruck, Austria}
\date{}
\maketitle
\abstract{In this paper, weakly homogeneous generalized functions in the special Colom\-beau algebras are determined up to equality in the sense of generalized distributions. This yields characterizations that are formally similar to distribution theory. Further, we give several characterizations of equality in the sense of generalized distributions in these algebras.}

\section{Introduction}
In \cite{homog}, homogeneity in algebras of generalized functions on $\R^d$ and on $\R^d\setminus\{0\}$ is investigated. Such algebras have been developed by many authors \cite{AR, DHPV04, GKOS, NPS} mainly inspired by the work of J.-F.\ Colombeau \cite{Colombeau84}, and have proved valuable as a tool for treating partial differential equations with singular data or coefficients (see \cite{MO92} and the references therein). In \cite{homog}, the attention is focused on the special Colombeau algebra: the existence of embeddings of the space of distributions $\test'$ with optimal consistency properties into this algebra \cite[\S 1.1--1.2]{GKOS} allows to compare homogeneity of generalized functions with the distributional homogeneity.

A result of this investigation is that the class of generalized functions (called \emph{strongly homogeneous}) satisfying a homogeneous equation in the sense of the usual equality in the algebra, is surprisingly restrictive: on the space $\R^d$, the only strongly homogeneous generalized functions are polynomials with generalized coefficients. Hence the embedded images of homogeneous distributions fail in general to be strongly homogeneous. For that reason, generalized functions (called \emph{weakly homogeneous}) satisfying a homogeneous equation \emph{in the sense of generalized distributions} \cite[\S 7.5]{Colombeau84}, i.e., when acting on (smooth, compactly supported, non-generalized) test functions, were considered in \cite{homog} and were shown to include the embedded images of homogeneous distributions.

The purpose of this paper is to characterize and study in detail the weakly homogeneous generalized functions.

We now describe our main results. On $\R$, the weakly homogeneous generalized functions of degree $\alpha\in\R\setminus\{-1,-2,\dots\}$ are, up to equality in the sense of generalized distributions, of the form $c_1\iota(x^\alpha_-) + c_2\iota(x^\alpha_+)$,
where $\iota$ denotes an embedding of $\test'$ into the Colombeau algebra and $c_1, c_2$ are generalized constants,
and those weakly homogeneous of degree $\alpha\in\{-1,-2,\dots\}$ are, up to equality in the sense of generalized distributions, of the form
$c_1 D^{-\alpha-1} \iota(\delta) + c_{2} \iota(x^\alpha)$
(Theorem \ref{thm_weak_hom_char_dim1}).\\
A weakly homogeneous generalized function $f$ on $\R^d\setminus\{0\}$ is, up to equality in the sense of generalized distributions, of the form
$g({x}/{\abs{x}})\abs{x}^\alpha$,
where $g\in\Gen(\R^d\setminus\{0\})$ is a radial mean of $f$ (Theorem \ref{thm_weak_hom_char_dimd}).
Further, a generalized function is shown to be weakly homogeneous if and only if it satisfies the corresponding Euler differential equation in the sense of generalized distributions (Theorem \ref{thm_weak_Euler}).

Let us emphasize that weakly homogeneous generalized functions are not assumed to be associated with a distribution \cite[\S 1.2.6]{GKOS}; hence these results cannot be obtained as a consequence of distribution theory (moreover, we will show that certain properties of distributions usually used to characterize homogeneous distributions do not hold in this more general setting (Examples \ref{ex_density}, \ref{ex_weak_point_support})). Instead, we develop other techniques using the uniform boundedness principle and properties of the Fourier transform in sections \ref{section_uniform_bdedness} and \ref{section_Fourier}. This allows us to obtain characterizations for the equality in the sense of generalized distributions (Theorem \ref{thm_strongly_associated}, Theorem \ref{thm_weakly_zero_char}).
We also indicate that some of our results can be obtained in more general sequence space algebras of generalized functions \cite{DS88,DHPV04}.

\section{Preliminaries}
\begin{df}
We call a sequence $(a_n)_{n\in\N}$ of maps $(0,1]\to \R^+$ an \defstyle{(asymptotic) scale} if
\begin{gather}
(\forall c\in\R^+) (\exists N\in\N) (a_N(\eps)\le c, \text{ for small }\eps)\\
(\forall n\in\N) (a_{n+1}(\eps)\le a_n(\eps), \text{ for small }\eps)\\
(\forall n,m\in\N) (\exists N\in\N) (a_N(\eps) \le a_n(\eps) a_m(\eps), \text{ for small } \eps).
\end{gather}
\end{df}
This definition is a slight generalization of the definition of asymptotic scale in \cite{DS88} with the purpose to also allow $a_n=1/n$ as a scale.
We also adopt the notation $a_{-n}(\eps):=1/a_n(\eps)$.

\begin{df}
Let $E$ be a topological vector space and $(a_n)_{n\in\N}$ a scale. Following \cite{DS88}, the set $\Mod_a(E)$ of \defstyle{$a_n$-moderate} nets in $E^{(0,1]}$ is defined as the set of those $(u_\eps)_\eps\in E^{(0,1]}$ for which
\[
(\forall U \text{ neighb.\ of $0$ in } E) (\exists N\in\N) (u_\eps \in {a_{-N}(\eps)} U, \text{ for small }\eps).
\]
In particular, we call $1/n$-moderateness also \defstyle{asymptotic boundedness} (since it is closely related to the notion of boundedness for subsets of a topological vector space).\\
The set $\Null_a(E)$ of \defstyle{$a_n$-negligible} (or: \defstyle{$a_n$-rapidly decreasing}) nets in $E^{(0,1]}$ is defined as the set of those $(u_\eps)_\eps\in E^{(0,1]}$ for which
\[
(\forall U \text{ neighb.\ of $0$ in } E) (\forall m\in\N) (u_\eps \in a_m(\eps) U, \text{ for small }\eps).
\]
\end{df}
In particular, $1/n$-negligibility coincides with convergence to $0$.
Since the notions of $a_n$-moderateness and $a_n$-negligibility remain unchanged when each $a_n$ is changed on an interval $[\eps_n,1]$ ($\eps_n>0$) and when $(a_n)_{n\in\N}$ is replaced by a subsequence, we can always find an equivalent scale $(b_n)_{n\in\N}$ of maps $(0,1]\to (0,1]$ such that for each $m,n\in\N$, $b_{n+m} \le b_n b_m$.
This will be silently assumed when it is allowed to consider an equivalent scale. It is also sufficient to test moderateness and negligibility for a base of neighbourhoods of $0$ only. In particular, if $E$ is locally convex, with a family of seminorms $(p_i)_{i\in I}$ describing its topology, then 
\begin{align*}
\Mod_a(E)&=\{(u_\eps)_\eps\in E^{(0,1]}: (\forall i\in I) (\exists N\in\N) (p_i(u_\eps)\le a_{-N}(\eps),\text{ for small } \eps)\}\\
\Null_a(E)&=\{(u_\eps)_\eps\in E^{(0,1]}: (\forall i\in I) (\forall m\in\N) (p_i(u_\eps)\le a_{m}(\eps),\text{ for small } \eps)\}.
\end{align*}
If $(a_n)=(\eps^n)$, we simply write $\Mod(E)$ for $\Mod_a(E)$, $\Null(E)$ for $\Null_a(E)$ and $\Gen_E:=\Mod(E)/\Null(E)$ is the Colombeau $\GenC$-module based on $E$ \cite[\S 3]{Garetto2005}. The element $u\in\Gen_E$ with representative $(u_\eps)_\eps\in\Mod(E)$ is sometimes denoted by $[(u_\eps)_\eps]$. Further, $\Gen_E^\infty:=\Mod^\infty(E)/\Null(E)\subset \Gen_E$ \cite[Ex.~3.10]{Garetto2005}, where
\[\Mod^\infty(E) = \{(u_\eps)_\eps\in E^{(0,1]}: (\exists N\in\N) (\forall i\in I) (p_i(u_\eps)\le\eps^{-N}, \text{ for small }\eps)\}.\]
Let $\Omega\subseteq\R^d$ open. If we choose in particular $E=\Cnt[\infty](\Omega)$ with its usual locally convex topology, then $\Gen_E=\Gen(\Omega)$ is the special Colombeau algebra of generalized functions on $\Omega$ \cite[Ex.~3.6]{Garetto2005}. Usually, $\Mod(E)$ is then denoted by $\mathcal E_M(\Omega)$.
The subalgebra $\Gen_c(\Omega)$ of compactly supported generalized functions is the set of those $u\in \Gen(\Omega)$ having a representative $(u_\eps)_\eps$ with $(\supp u_\eps)_\eps\in K^{(0,1]}$, for some $K\csubset\Omega$. 
If we choose $E=\C$, $\Gen_E=\GenC$ is the ring of Colombeau generalized constants. For $f\in\Gen_c(\Omega)$, $\int_\Omega f\in\GenC$, defined on representatives by $(\int_\Omega f_\eps)_\eps\in\Mod(\C)$, is well-defined. $\Gen^\infty(\Omega)$ is defined as the algebra of those $u\in\Gen(\Omega)$ having a representative $(u_\eps)_\eps$ for which 
\[(\forall K\csubset \Omega) (\exists N\in\N) (\forall\alpha\in\N^d) (\sup_{x\in K}\abs{\partial^\alpha u_\eps(x)}\le\eps^{-N}, \text{ for small }\eps).\]
Similarly, $\Gen^\infty_c(\Omega) := \Gen^\infty(\Omega)\cap \Gen_c(\Omega)$. Denoting by $\test(\Omega)$ the space of compactly supported $\Cnt[\infty]$-functions on $\Omega$, $f,g\in\Gen(\Omega)$ are called equal in the sense of generalized distributions if for each $\phi\in\test(\Omega)$, $\int_\Omega f \phi=\int_\Omega g \phi$.\\
If we choose $E=\R^d$, $\Gen_E=\GenR^d$ is the set of $d$-dimensional generalized points. For an open set $\Omega\subseteq\R^d$, the set $\widetilde\Omega_c$ of compactly supported generalized points of $\Omega$ is the set of those $x\in\GenR^d$ having a representative $(x_\eps)_\eps\in K^{(0,1]}$, for some $K\csubset\Omega$. In particular, for $\Omega=\R^+=\{x\in\R: x >0\}$, we denote $\widetilde\Omega_c$ by $\GenR^+_c$.
For $u\in\Gen(\Omega)$ and $x\in\widetilde\Omega_c$, the point value $u(x)\in\GenC$, defined on representatives by $(u_\eps(x_\eps))_\eps\in\Mod(\C)$, is well-defined \cite[1.2.45]{GKOS}. By a similar Taylor-argument, for $u\in\Gen_c(\R^d)$ and $\xi\in\GenR^d$, the Fourier transform $\widehat u(\xi)=\int_{\R^d} u(x) e^{-i\xi x}\,dx\in\GenC$ is well-defined. Moreover, $u\in\Gen(\Omega)$ is completely determined by its point values in compactly supported points \cite[Thm.~1.2.46]{GKOS}. There exist embeddings $\iota$ of the space of distributions $\test'(\Omega)$ into $\Gen(\Omega)$ that preserve the linear operations, derivatives, the product of $\Cnt[\infty](\Omega)$-functions and the pairing (i.e., for $T\in\test'(\Omega)$ and $\phi\in\test(\Omega)$, $\int_\Omega \iota(T) \phi =\pairing{T}{\phi}$ in $\GenC$).
For further properies of $\Gen(\Omega)$, we refer to \cite{GKOS}.\\
We refer to \cite{Treves} for definitions related to topological vector spaces. In particular, a barreled topological vector space is not assumed to be locally convex. We refer to \cite{Estrada} for the definitions of the regularized distributions $x^{-m}$, $\FinPart{H(x)/x^m}$, $x^\alpha_+$, $x^\alpha_-$ $\in\test'(\R)$ (for $m\in\N$, $m\ge 1$ and $\alpha\in\R$, $\alpha\ne -1, -2,\dots$).
We will also denote $S^{d-1}=\{x\in\R^d: \abs{x}=1\}$.

\section{Uniform boundedness and weak equalities}\label{section_uniform_bdedness}
In a way similar to the use in \cite{PSV} of the Baire theorem, we exploit the uniform boundedness principle to show that an equality in the sense of generalized distributions automatically holds for test functions in the larger space $\Gen_c^\infty(\Omega)$. We cast our results in a general framework for two reasons:\\
(1) In this way, it is clear that our results for Colombeau algebras also hold for more general sequence space algebras \cite{DS88,DHPV04}.\\
(2) Using the framework of Colombeau algebras based on a locally convex vector space $E$ \cite{Garetto2005}, our results can be applied to other spaces than $\test(\Omega)$ (e.g., this is already needed for Proposition \ref{prop_weakly_0_outside_origin}).

\begin{lemma}\label{lemma_subsequence_of_asympt_bd}
Let $E$ be a topological vector space. Let $(\eps_n)_{n\in\N}\in (0,1]^\N$ with $\eps_n\to 0$ as $n\to\infty$.
\begin{enumerate}
\item If $(u_\eps)_\eps\in E^{(0,1]}$ is asymptotically bounded, then $\{u_{\eps_n} :n\in\N\}$ is a bounded set in $E$.
\item Let $(a_n)_{n\in\N}$ be a scale. If $(u_\eps)_\eps\in\Mod_a(E)$ and $(m_n)_{n\in\N}\in \N^\N$ tends to infinity, then $\{a_n(\eps_{m_n}) u_{\eps_{m_n}} : n\in\N\}$ is a bounded set in $E$.
\item Let $(a_n)_{n\in\N}$ be a scale. If $(u_\eps)_\eps\in\Null_a(E)$ and $M\in\N$, then $\{a_{-M}(\eps_n) u_{\eps_{n}} : n\in\N\}$ is a bounded set in $E$.
\end{enumerate}
\end{lemma}
\begin{proof}
(1) Let $U$ be a balanced neighbourhood of $0$ in $E$. As $(u_\eps)_\eps\in E^{(0,1]}$ is asymptotically bounded and $\eps_n\to 0$, there exists $M\in\N$ and $\lambda\in\R^+$ such that $\{u_{\eps_n}: n\ge M\}\subseteq \lambda U$. Then there also exists $\lambda'\in\R^+$ such that the union of $\{u_{\eps_n}: n\ge M\}$ with the finite set $\{u_{\eps_n}: n < M\}$ is contained in $\lambda' U$.\\
(2) Let $U$ be a balanced neighbourhood of $0$ in $E$. Then there exist $M\in\N$ and $N\in\N$ such that for each $n\ge M$, $u_{\eps_{m_n}}\in a_{-N}(\eps_{m_n}) U$. Hence $a_n(\eps_{m_n})u_{\eps_{m_n}}\in a_n(\eps_{m_n}) a_{-N}(\eps_{m_n}) U\subseteq U$, as soon as $n\ge\max(M,N)$. As in part~(1), this implies that there exists $\lambda\in\R^+$ such that $\{a_n(\eps_{m_n}) u_{\eps_{m_n}} : n\in\N\}\subseteq \lambda U$.\\
(3) By definition of $\Null_a(E)$, the net $(a_{-M}(\eps) u_\eps)_\eps$ is asymptotically bounded, $\forall M$. The result follows by part~1.
\end{proof}

\begin{thm}
\label{thm_asympt_uniform_bd}
Let $E$ be a barreled topological vector space. Let $(a_n)_{n\in\N}$ be a scale. Let $(T_\eps)_\eps\in\Mod_a(E')$ (for the topology of pointwise convergence), i.e.,
\[
(\forall u\in E) (\exists N\in\N) (\abs{T_\eps(u)}\le a_{-N}(\eps), \text{ for small }\eps).
\]
\begin{enumerate}
\item If $(u_\eps)_\eps\in\Mod_a(E)$, then also $(T_\eps(u_\eps))_\eps\in\Mod_a(\C)$.
\item Let $(a_n)=(1/n)$ or let the topology of $E$ have a countable base of neighbourhoods of $0$. If $(u_\eps)_\eps\in \Null_a(E)$, then also $(T_\eps(u_\eps))_\eps\in\Null_a(\C)$.
\item Let $E$ be locally convex with its topology generated by a sequence of seminorms $(p_k)_{k\in\N}$. Then
\[
(\exists M\in\N) (\exists\eps_0\in (0,1]) (\forall\eps\le\eps_0) (\forall u \in E)
\big(\abs{T_\eps(u)}\le a_{-M}(\eps) \max_{k\le M} p_k(u)\big).\]
\end{enumerate}
\end{thm}
\begin{proof}
(1) Supposing that $(T_\eps(u_\eps))_\eps\notin\Mod_a(\C)$, we find a decreasing sequence $(\eps_n)_{n\in\N}$ tending to $0$ such that $\abs{T_{\eps_n}(u_{\eps_n})}> a_{-n}(\eps_n)$, $\forall n$.
As $(T_\eps)_\eps\in\Mod_a(E')$, $\{a_n(\eps_{2n}) T_{\eps_{2n}}: n\in\N\}$ is bounded in $E'$ (with the topology of pointwise convergence) by lemma \ref{lemma_subsequence_of_asympt_bd}(2). Hence the uniform boundedness principle \cite[Thm.~33.1]{Treves} implies that $\{a_n(\eps_{2n}) T_{\eps_{2n}}:n\in\N\}$ is equicontinuous, i.e.,
\[
(\forall r\in\R^+) (\exists U \text{ neighb.\ of $0$ in }E) (\forall n\in\N) (\forall u\in U)(\abs{a_n(\eps_{2n}) T_{\eps_{2n}}(u)}\le r).
\]
Choose $r=1$ and choose a corresponding neighbourhood $U$.
Since $(u_\eps)_\eps\in\Mod_a(E)$, there exists $N\in\N$ such that $a_N(\eps_{2n}) u_{\eps_{2n}}\in U$, for sufficiently large $n$. But then $\abs{a_n(\eps_{2n}) T_{\eps_{2n}}(a_N(\eps_{2n}) u_{\eps_{2n}})}\le 1$, for sufficiently large $n$. Hence
\[
\abs{T_{\eps_{2n}}(u_{\eps_{2n}})}\le {a_{-n}(\eps_{2n}) a_{-N}(\eps_{2n})}\le {(a_{-n}(\eps_{2n}))^2}\le {a_{-2n}(\eps_{2n})},
\]
for sufficiently large $n$.
This contradicts $\abs{T_{\eps_n}(u_{\eps_n})}> a_{-n}(\eps_n)$, $\forall n$.\\
(2) First case: $(a_n)=(1/n)$. Supposing that $T_\eps(u_\eps)\not\to 0$, we find $\lambda\in\R^+$ and a decreasing sequence $(\eps_n)_{n\in\N}$ tending to $0$ such that $\abs{T_{\eps_n}(u_{\eps_n})}> \lambda$, $\forall n$. By lemma \ref{lemma_subsequence_of_asympt_bd}(1), $\{T_{\eps_n}:n\in\N\}$ is bounded in $E'$. As in part~(1), it follows that
\[
(\exists U \text{ neighb.\ of $0$ in }E) (\forall n\in\N) (\forall u\in U)(\abs{T_{\eps_n}(u)}\le \lambda)
\]
and since $u_\eps\to 0$, $u_{\eps_n}\in U$ for sufficiently large $n$. A contradiction follows.\\
Second case: let $(U_n)_{n\in\N}$ be a countable base of neighbourhoods of $0$ in $E$ with $U_{n+1}\subseteq U_n$, $\forall n$. As $(u_\eps)_\eps\in\Null_a(E)$, we find for each $m\in\N$ some $\eta_m\in (0,1]$ such that $u_\eps\in a_m(\eps)U_m$, $\forall\eps\le\eta_m$. Supposing that $(T_\eps(u_\eps))_\eps\notin\Null_a(\C)$, we find $M\in\N$ and a decreasing sequence $(\eps_n)_{n\in\N}$ tending to $0$ such that $\abs{T_{\eps_n}(u_{\eps_n})}> a_M(\eps_n)$, $\forall n$. We may suppose that $\eps_n\le \eta_n$, $\forall n$. Now let
\[
v_\eps =
\begin{cases}
a_{-M}(\eps_n) a_{-n}(\eps_n) u_{\eps_n},&\eps\in\{\eps_n:n\in\N\}\\
0,&\text{otherwise}.
\end{cases}
\]
Let $U$ be a neighbourhood of $0$ in $E$. Then there exists $N\in\N$ such that $U_N\subseteq U$. Hence, $\forall n\ge N$, as $\eps_n\le\eta_n$,
\[v_{\eps_n}= a_{-M}(\eps_n) a_{-n}(\eps_n) u_{\eps_n} \in a_{-M}(\eps_n) U_n \subseteq a_{-M}(\eps_n) U.\]
Hence $(v_\eps)_\eps\in \Mod_a(E)$. By part~(1), $(T_\eps(v_\eps))_\eps\in\Mod_a(\C)$. Further, $\forall n$,
\[
\abs{T_{\eps_n}(v_{\eps_n})}= a_{-M}(\eps_n) a_{-n}(\eps_n) \abs{T_{\eps_n}(u_{\eps_n})}> a_{-n}(\eps_n).
\]
This contradicts $(T_\eps(v_\eps))_\eps\in\Mod_a(\C)$.\\
(3) By contraposition, we find a decreasing sequence $(\eps_n)_{n\in\N}$ tending to $0$ and $u_n\in E$ such that $\abs{T_{\eps_n}(u_{n})}> a_{-n}(\eps_n) \max_{k\le n} p_k(u_n)$, $\forall n$.
By lemma \ref{lemma_subsequence_of_asympt_bd}(2), $\{a_n(\eps_{2n})T_{\eps_{2n}}:n\in\N\}$ is bounded in $E'$. Expressing the equicontinuity, we obtain
\[
(\exists C\in \R^+) (\exists N\in\N) (\forall u\in E) (\forall n\in\N) \big(a_n(\eps_{2n})\abs{T_{\eps_{2n}}(u)}\le C \max_{k\le N} p_k(u)\big).
\]
In particular, for sufficiently large $n\in\N$,
\begin{align*}
a_{-n}(\eps_{2n}) \max_{k\le N} p_k(u_{2n}) & \le a_n(\eps_{2n}) a_{-2n}(\eps_{2n}) \max_{k\le 2n} p_k(u_{2n})\\
&< a_n(\eps_{2n})\abs{T_{\eps_{2n}}(u_{2n})}\le C \max_{k\le N} p_k(u_{2n}),
\end{align*}
hence $a_{-n}(\eps_{2n}) < C$, for sufficiently large $n$, contradicting the fact that $a_n$ is a (w.l.o.g.\ decreasing in $n$)
scale.
\end{proof}

\begin{thm}\label{thm_(f_n)-asympt}
Let $E$ be a barreled topological vector space. Let $(a_n)_{n\in\N}$ be a scale. Let $(\forall m\in\N) (\exists n\in\N) \big(\lim\limits_{\eps\to 0} \frac{a_n(\eps)}{a_m(\eps)}=0\big)$ or let every bounded set in $E$ be precompact. Let $(T_\eps)_\eps\in \Null_a(E')$ (for the topology of pointwise convergence). If $(u_\eps)_\eps\in E^{(0,1]}$ is asymptotically bounded, then also $(T_\eps(u_\eps))_\eps\in\Null_a(\C)$.
\end{thm}
\begin{proof}
Supposing that $(T_\eps(u_\eps))_\eps\notin\Null_a(\C)$, we find $M\in\N$ and a decreasing sequence $(\eps_n)_{n\in\N}$ tending to $0$ such that $\abs{T_{\eps_n}(u_{\eps_n})}> a_M(\eps_n)$, $\forall n$. If every bounded set in $E$ is precompact, choose $M'=M$; otherwise, choose $M'$ such that $\lim_{\eps\to 0} \frac{a_{M'}(\eps)}{a_M(\eps)}=0$. By lemma~\ref{lemma_subsequence_of_asympt_bd}(3), $\{a_{-M'}(\eps_n) T_{\eps_n}: n\in\N\}$ is bounded in $E'$ (for the topology of pointwise convergence). As in proposition \ref{thm_asympt_uniform_bd}, we find a neighbourhood $U$ of $0$ in $E$ such that
\[
(\forall n\in\N) (\forall u\in U) (\abs{T_{\eps_n}(u)}\le a_{M'}(\eps_n)/2).
\]
By lemma \ref{lemma_subsequence_of_asympt_bd}(1), $\{u_{\eps_n}: n\in\N\}$ is bounded.\\
First case: every bounded set in $E$ is precompact. Then we find a finite subset $F\subseteq E$ such that \[(\forall n\in\N) (\exists v\in F) (u_{\eps_n} - v \in U),\]
hence also
\[(\forall n\in\N) (\exists v\in F) (\abs{T_{\eps_n}(u_{\eps_n}-v)}\le a_M(\eps_n)/2).\]
Further, since $(T_\eps)_\eps\in\Null_a(E')$ (for the topology of pointwise convergence),
\[(\exists \eps_0\in (0,1]) (\forall \eps\le \eps_0) (\forall v\in F) (\abs{T_\eps(v)}\le a_M(\eps)/2).\]
Combining these identities,
\[(\exists N\in\N) (\forall n\ge N) (\exists v\in F)(\abs{T_{\eps_n}(u_{\eps_n})} \le \abs{T_{\eps_n}(u_{\eps_n}-v)} + \abs{T_{\eps_n}(v)}\le a_M(\eps_n)).\]
This contradicts $\abs{T_{\eps_n}(u_{\eps_n})}> a_M(\eps_n)$, $\forall n$.\\
Second case: as $\{u_{\eps_n}: n\in\N\}$ is bounded, there exists $\lambda\in\R^+$ such that $u_{\eps_n}\in \lambda U$, $\forall n$, hence also
\[\abs{T_{\eps_n}(u_{\eps_n})}\le \lambda a_{M'}(\eps_n)/2 \le a_M(\eps_n),\]
for sufficiently large $n$, again contradicting $\abs{T_{\eps_n}(u_{\eps_n})}> a_M(\eps_n)$, $\forall n$.
\end{proof}

\begin{prop}\label{prop_inductive_limit}
Let $E$ be a strict inductive limit 
of Fr{\'e}chet spaces $(E_n)_{n\in\N}$.
\begin{enumerate}
\item For each $n\in\N$, the identity map on representatives $\Mod(E_n)\to\Mod(E)$ induces a canonical embedding $\Gen_{E_n}\to\Gen_{E}$.
\item Identifying $\Gen_{E_n}$ with its image under the embedding in part~(1), $\Gen_E=\bigcup_{n\in\N}\Gen_{E_n}$.
\item For each $n\in\N$, $\Gen_{E_n}\cap \Gen_E^\infty= \Gen_{E_n}^\infty$
(hence also $\Gen_E^\infty=\bigcup_{n\in\N}\Gen_{E_n}^\infty$).
\end{enumerate}
\end{prop}
\begin{proof}
(1), (3) For each $n\in\N$, $E_n^{(0,1]}\cap \Mod(E)=\Mod(E_n)$, $E_n^{(0,1]}\cap\Null(E) = \Null(E_n)$ and $E_n^{(0,1]}\cap \Mod^\infty(E)=\Mod^\infty(E_n)$ follow easily since a convex $U\subseteq E$ is a neighbourhood of $0$ in $E$ iff $U\cap E_n$ is a neighbourhood of $0$ in $E_n$, for each $n$ and because the topology on $E_n$ is the relative topology induced by $E$.\\
(2) Let $(u_\eps)_\eps\in\Mod(E)$. We show that there exists $m\in\N$ and $\eps_0\in(0,1]$ such that for each $\eps\le\eps_0$, $u_\eps\in E_m$. For, supposing the contrary, we find a decreasing sequence $(\eps_n)_{n\in\N}$ tending to $0$ such that $u_{\eps_n}\notin E_n$, for each $n$. As in the proof of \cite[Prop.~14.6]{Treves}, one can inductively construct convex neighbourhoods $U_n$ of $0$ in $E_n$ such that $U_{n+1}\cap E_n= U_n$ and $\eps_j^j u_{\eps_j}\notin U_n$, for each $j<n$. Then $U=\bigcup_{n\in\N} U_n$ is a convex neighbourhood of $0$ in $E$ and for each $n\in\N$, $u_{\eps_n}\notin\eps_n^{-n} U$, contradicting $(u_\eps)_\eps\in \Mod(E)$.
So if $u\in\Gen_E$, we find $m\in\N$ and a representative $(u_\eps)_\eps$ of $u$ such that $(u_\eps)_\eps\in E_m^{(0,1]}\cap \Mod(E)=\Mod(E_m)$.
\end{proof}

\begin{thm}\label{thm_strongly_associated}
Let $\Omega\subseteq\R^d$ open. Let $f\in\Gen(\Omega)$ and $T\in\test'(\Omega)$. Let $\int_\Omega f \phi = T(\phi)$, $\forall\phi\in\test(\Omega)$. Then $\int_\Omega f\phi = T(\phi)$, $\forall\phi\in\Gen_c^\infty(\Omega)$.
\end{thm}
\begin{proof}
By the hypotheses, the net $(\phi\mapsto \int_\Omega f_\eps\phi - T(\phi))_\eps\in\Null(\test'(\Omega))$. Let $\phi\in\Gen^\infty_c(\Omega)$. By 
proposition \ref{prop_inductive_limit}(3), for a compact exhaustion $(K_n)_{n\in\N}$ of $\Omega$, $\Gen_c^\infty(\Omega)=\bigcup_{n\in\N}\Gen^\infty_{\test(K_n)}=\Gen_{\test(\Omega)}^\infty$. Then there exists $N\in\N$ and a representative $(\phi_\eps)_\eps$ of $\phi$ such that the net $(\eps^N\phi_\eps)_\eps$ is asymptotically bounded. As $\test(\Omega)$ is barreled \cite[Prop.~34.4]{Treves}, the net $(\eps^N (\int_\Omega f_\eps\phi_\eps - T(\phi_\eps)))_\eps\in\Null(\test(\Omega))$ by theorem \ref{thm_(f_n)-asympt}.
\end{proof}

\section{Characterization of equality in the sense of generalized distributions by means of the Fourier transform}\label{section_Fourier}

\begin{df}
A generalized point $[(x_\eps)_\eps]\in\GenR^d$ is said to be \defstyle{of slow scale} if for each $a\in\R^+$, 
$\abs{x_\eps}\le\eps^{-a}$, for small $\eps$. We denote the set of slow scale points by $\GenR^d_{ss}$.
\end{df}
\begin{prop}\label{prop_regular_Fourier_inverse}
Let $f\in\Gen_c^\infty(\R^d)$. Then $f=0$ iff $\widehat f(\xi) = 0$, $\forall\xi\in\GenR^d_{ss}$.
\end{prop}
\begin{proof}
If $f=0$, then clearly $\widehat f(\xi) = 0$, $\forall\xi\in\GenR^d$. Conversely, let $\widehat f(\xi) = 0$, $\forall\xi\in\GenR^d_{ss}$. Let $(f_\eps)_\eps$ be a representative of $f$ with $\supp f_\eps\subseteq K\csubset\R^d$, $\forall\eps$. Since $f\in\Gen_c^\infty(\R^d)$, there exists $M\in\N$ such that for each $\beta\in\N^d$,
\begin{equation*}
\sup_{\xi\in\R}\abs{\xi}^\beta\abs[\big]{\widehat f_\eps(\xi)} =\sup_{\xi\in\R}\abs[\big]{\widehat{\partial^\beta f_\eps}(\xi)}
\le \meas(K) \sup_{x\in K}\abs[\big]{\partial^\beta f_\eps(x)}\le \eps^{-M},\text{ for small }\eps.
\end{equation*}
We show that for each $k\in\N$, there exists $N\in\N$ such that
\begin{equation}\label{eqn_Fourier-overspill}
\sup_{\abs{\xi}\le\eps^{-1/{N}}} \abs[\big]{\widehat{f_\eps}(\xi)}\le\eps^k, \text{ for small }\eps.
\end{equation}
Supposing the contrary, we would find $k\in\N$, a decreasing sequence $(\eps_n)_{n\in\N}$ tending to $0$ and $\xi_{\eps_n}$ with $\abs{\xi_{\eps_n}}\le\eps_n^{-1/n}$ and $\abs[\big]{\widehat{f_{\eps_n}}(\xi_{\eps_n})}>\eps_n^k$, $\forall n$. Defining $\xi_\eps=0$ for $\eps\notin\{\eps_n:n\in\N\}$, the net $(\xi_\eps)_\eps$ would represent some $\xi\in\GenR^d_{ss}$; yet $\widehat{f}(\xi)\ne 0$, contradicting the hypotheses.\\
Now let $k\in\N$ arbitrary. Choose $N\in\N$ as in equation (\ref{eqn_Fourier-overspill}). Then for each $[(x_\eps)_\eps]\in\GenR^d_c$,
\begin{multline*}
(2\pi)^d\abs{f_\eps(x_\eps)} = \abs[\Big]{\int_{\R^d} \widehat{f_\eps}(\xi) e^{i\xi x_\eps}\,d\xi}
\le \int_{\abs{\xi}\le \eps^{-\frac{1}{N}}} \abs[\big]{\widehat{f_\eps}(\xi)}\,d\xi + \int_{\abs{\xi}\ge \eps^{-\frac{1}{N}}} \abs[\big]{\widehat{f_\eps}(\xi)}\,d\xi\\
\le C \eps^k \eps^{-\frac{d}{N}} + C\eps^{k}\sup_{\xi\in\R}\abs{\xi}^{Nk+d+1}\abs[\big]{\widehat f_\eps(\xi)}\le C\eps^{k-M-d},
\end{multline*}
for small $\eps$ (for some $C\in\R^+$). Since $k\in\N$ and $x\in\GenR^d_c$ are arbitrary, $f=0$.
\end{proof}

\begin{thm}\label{thm_weakly_zero_char}
Let $f\in\Gen_c(\R^d)$. The following are equivalent:
\begin{enumerate}
\item $\displaystyle\int_{\R^d}f\phi=0, \quad\forall\phi\in\test(\R^d)$
\item $\displaystyle\int_{\R^d}f\phi=0, \quad\forall\phi\in\Gen_c^\infty(\R^d)$
\item $\displaystyle\widehat f(\xi)=0, \quad\forall\xi\in\GenR^d_{ss}$.
\end{enumerate}
\end{thm}
\begin{proof}
$(1)\Rightarrow (2)$: by theorem \ref{thm_strongly_associated}.\\
$(2)\Rightarrow (3)$: let $\xi=[(\xi_\eps)_\eps]\in\GenR^d_{ss}$.
For $\beta\in\N^d$,
$\sup_{x\in \R^d}\abs[\big]{\partial^\beta_x e^{-i\xi_\eps x}} \le \abs{\xi_\eps}^{\abs\beta} \le \eps^{-1}$, for small $\eps$. Hence $x\mapsto e^{-i\xi x}\in \Gen^\infty(\R^d)$. As $f\in\Gen_c(\R^d)$, there exists $\psi\in\test(\R^d)$ with $f\psi=f$. Since $\psi e^{-i\xi x}\in\Gen_c^\infty(\R^d)$, the hypotheses imply that
\[\widehat f(\xi)=\int_{\R^d}f(x)\psi(x)e^{-i\xi x}\,dx=0.\]
$(3)\Rightarrow(1)$: let $\phi\in\test(\R^d)$. Then $f\conv\phi\in\Gen_c^\infty(\R^d)$, since with $K=\supp f\csubset \R^d$,
\[
\sup_{x\in\R^d}\abs[\big]{\partial^\beta (f_\eps\conv\phi)(x)}
\le\meas(K) \sup_K\abs{f_\eps}\sup_{\R^d}\abs[\big]{\partial^\beta\phi},
\]
where $\meas$ denotes the Lebesgue measure. Since for each $\xi\in\GenR^d$, $\widehat{f\conv\phi}(\xi) =\widehat f (\xi)\widehat \phi(\xi)$, we obtain by proposition \ref{prop_regular_Fourier_inverse} that $f\conv\phi=0$. In particular, $(f\conv\phi)(0)=\int_{\R^d}f(y)\phi(-y)\,dy=0$, $\forall\phi\in\test(\R^d)$.
\end{proof}
Notice that proposition \ref{prop_regular_Fourier_inverse} and theorem \ref{thm_weakly_zero_char} are consistent with the statement in \cite{PSV} that for $\Omega\subseteq\R^d$ open and $f\in\Gen^\infty(\Omega)$, $f=0$ iff $\int_{\Omega}f\phi=0$, $\forall\phi\in\test(\Omega)$.

\begin{cor}\label{cor_weak_equal_coordwise_test_functions}
Let $f\in\Gen(\R^d)$. If
\begin{equation}\label{eqn_componentwise}
\int_{\R^d}f(x)\phi_1(x_1)\cdots\phi_d(x_d)\,dx=0, \quad\forall \phi_1,\dots,\phi_d\in\test(\R),
\end{equation}
then $\int_{\R^d}f\phi=0$, $\forall\phi\in\test(\R^d)$.
\end{cor}
\begin{proof}
Fix $\phi_2$,\dots, $\phi_d$ $\in\test(\R)$ and let $F(x_1)=\int_{\R^{d-1}}f(x)\phi_2(x_2)\cdots\phi_d(x_d)dx_2\dots dx_d$. Then $F\in\Gen(\R)$ and by hypothesis, $\int_\R F(x_1)\phi(x_1)\,dx_1=0$, $\forall\phi\in\test(\R)$. By theorem \ref{thm_strongly_associated}, the same holds for each $\phi\in\Gen_c^\infty(\R)$. Inductively, equality (\ref{eqn_componentwise}) holds for each $\phi_1$, \dots, $\phi_d\in\Gen_c^\infty(\R)$.\\
Now fix $\phi\in\test(\R^d)$. Then there exist $\psi_1,\dots,\psi_d\in\test(\R)$ such that $\phi=\psi\phi$, where $\psi(x)=\psi_1(x_1)\cdots\psi_d(x_d)\in\test(\R^d)$. Let $\xi\in\GenR^d_{ss}$. Since $x_j\mapsto e^{-i\xi_j x_j}\in\Gen^\infty(\R)$, for $j=1,\dots,d$, it follows that
\[\widehat {f\psi}(\xi)=\int_{\R^d}f(x)(\psi_1(x_1)e^{-i\xi_1 x_1})\cdots(\psi_d(x_d)e^{-i\xi_d x_d})\,dx = 0.\]
As $f\psi\in\Gen_c(\R^d)$, $\int_{\R^d} f\psi\phi=\int_{\R^d} f\phi =0$ by theorem \ref{thm_weakly_zero_char}.
\end{proof}

The following example shows that there is no analogue in $\Gen(\R^d)$ to the usual distributional argument that, if $A$ is a dense subset of $\test(\R^d)$ and $T\in\test'(\R^d)$ satisfies $\pairing{T}{\phi}=0, \forall \phi\in A$, then $T=0$.

\begin{example}\label{ex_density}
Let $A\subseteq\test(\R^d)$ be a countable, dense set. Then there exists $f\in\Gen(\R^d)$ such that
\begin{enumerate}
\item $\int_{\R^d}f\phi=0$, $\forall \phi\in A$
\item there exists $\phi_0\in\test(\R^d)$ such that $\int_{\R^d}f\phi_0 = 1$.
\end{enumerate}
\end{example}
\begin{proof}
Let $A=\{\phi_1,\phi_2,\dots\}$.
Since $\test(\R^d)$ is of uncountable dimension (it contains $\test([0,1]^d)$, which is of uncountable dimension as a Baire topological vector space of infinite dimension), there exists $\phi_0\in\test(\R^d)\setminus\lspan A$. Let $m\in\N$. By a linear algebra argument \cite[Lemma V.3.10]{DS58}, there exists $g_m\in\Cnt[\infty](\R^d)$ such that $\int_{\R^d}g_m\phi_0=1$, $\int_{\R^d}g_m\phi_j=0$, for $j=1,\dots,m$. Now choose a decreasing sequence $(\eps_n)_{n\in\N}\in (0,1]^\N$ tending to $0$ such that $\sup_{\abs{x}\le n, \abs\alpha\le n} \abs{\partial^\alpha g_n(x)}\le\eps_n^{-1}$, for each $n\in\N$. Let $f_\eps = g_n$, for each $\eps\in (\eps_{n+1},\eps_n]$, for each $n\in\N$. Since $\sup_{\abs{x}\le n,\abs\alpha\le n}\abs{\partial^\alpha f_\eps(x)}\le\eps^{-1}$, $\forall\eps\le\eps_n$, the net $(f_\eps)_\eps$ is moderate, hence represents $f\in\Gen(\R^d)$ with $\int_{\R^d} f \phi_0=1$ and $\int_{\R^d} f\phi_n=0$, $\forall n\ge 1$.
\end{proof}

\begin{prop}\label{prop_weak_equal_polar_test_functions}
Let $f\in\Gen(\R^d\setminus\{0\})$. If
\[\int_{\R^d} f(x) u(\abs x)\,v\left(\frac{x}{\abs x}\right)\,dx=0,\quad \forall u\in\test(\R^+), \forall v\in\test(\R^d),\]
then $\int_{\R^d} f\phi=0$, $\forall \phi\in\test(\R^d\setminus\{0\})$.
\end{prop}
\begin{proof}
Let $(U_j,\alpha_j)_{j=1,\dots,m}$ be a finite $\Cnt[\infty]$-atlas for the unit sphere $S^{d-1}=\{x\in \R^d:\abs x=1\}$ and let $(g_j)_j$ be a $\Cnt[\infty]$-partition of unity of $S^{d-1}$ subordinate to the cover $(U_j)_j$. By assumption, for $u\in\test(\R^+)$, $v\in\test(\R^d)$ and $j\in\{1,\dots,m\}$,
\[\int_{\R^d}f(x)u(\abs x)\,v\left(\frac{x}{\abs x}\right)g_j\left(\frac{x}{\abs x}\right)\,dx=0.\]
Denoting the Jacobian of the transformation $\alpha_j^{-1}$ by $J_{\alpha_j^{-1}}$ and the local coordinates by $\Theta=(\theta_1,\dots,\theta_{d-1})$, we obtain
\[\int_0^{\infty}\int_{\alpha_j(U_j)} f(r\alpha_j^{-1}(\Theta)) u(r)v(\alpha_j^{-1}(\Theta)) g_j(\alpha_j^{-1}(\Theta)) r^{d-1} \abs[]{J_{\alpha_j^{-1}}(\Theta)}\,dr d\Theta=0.\]
Let $F(r,\Theta)=f(r\alpha_j^{-1}(\Theta))g_j(\alpha_j^{-1}(\Theta))u(r)$. Since $\supp F\csubset \R^+\times \alpha_j(U_j)$, we have for each $\phi\in\test(\R)$ and $\Phi\in\test(\R^{d-1})$ that
$\int_{\R^d}F(r,\Theta)\phi(r)\Phi(\Theta)\,drd\Theta=0$.
By corollary \ref{cor_weak_equal_coordwise_test_functions}, for each $\psi\in\test(\R^d)$, also
$\int_{\R^d} F(r,\Theta)\psi(r,\Theta)\,dr d\Theta=0$.
Since $u\in\test(\R^+)$ is arbitrary, this implies that for each $\phi\in\test(\R^d\setminus\{0\})$,
\[\int_0^{\infty}\int_{\alpha_j(U_j)} f(r\alpha_j^{-1}(\Theta)) \phi(r\alpha_j^{-1}(\Theta)) g_j(\alpha_j^{-1}(\Theta)) r^{d-1} \abs[]{J_{\alpha_j^{-1}}(\Theta)}\,dr d\Theta=0.\]
Hence
\[\int_{\R^d}f\phi=\sum_{j=1}^m\int_{\R^d}f(x)\phi(x) g_j\left(\frac{x}{\abs x}\right)\,dx=0, \quad\forall\phi\in\test(\R^d\setminus\{0\}).
\]
\end{proof}

\section{Weak homogeneity in $\Gen(\R)$ and $\Gen(\R\setminus\{0\})$}
Let $\Omega=\R^d$ or $\R^{d}\setminus\{0\}$.\\
Let $f\in\Gen(\Omega)$ and $\alpha\in\R$.
Following \cite{homog}, $f$ is called weakly homogeneous of degree $\alpha$ iff
\begin{equation}\label{eqn_weak_homog}
\int_{\R^d}f(\lambda x)\phi(x)\,dx=\int_{\R^d} \lambda^\alpha f(x)\phi(x)\,dx
\end{equation}
for each $\lambda\in\R^+$ and each $\phi\in\test(\Omega)$.
By theorem~\ref{thm_strongly_associated}, this then also holds for each $\lambda\in\R^+$ and each $\phi\in\Gen_c^\infty(\Omega)$.

\begin{lemma}\label{lemma_F}
Let $f\in\Gen(\Omega)$, $\alpha\in\R$ and $\psi\in\Gen_c^\infty(\Omega)$. If equation (\ref{eqn_weak_homog}) holds for each $\lambda\in\R^+$ and each $\phi\in\{x\mapsto \psi(\mu x): \mu\in\GenR^+_c\}\subset\Gen_c^\infty(\Omega)$, then the map $F:$ $\lambda\mapsto\int_{\R^d}f(\lambda x)\psi(x)\,dx$ defines a (strongly) homogeneous generalized function of degree $\alpha$ in $\Gen(\R^+)$.
\end{lemma}
\begin{proof}
Fix representatives $(f_\eps)_\eps$ of $f$ and $(\psi_\eps)_\eps$ of $\psi$. Consider for each $\eps\in(0,1]$, $F_\eps\in\Cnt[\infty](\R^+)$: $F_\eps(\lambda)=\int_{\R^d}f_\eps(\lambda x)\psi_\eps(x)\,dx$. Since for each $K\csubset\R^+$ and $m\in\N$,
\begin{multline*}
\sup_{\lambda\in K} \abs{D^m F_\eps(\lambda)}
=\sup_{\lambda\in K} \abs[\Bigg]{\int_{\R^d}\sum_{i_1,\dots,i_m=1}^d\partial_{i_1}\dots\partial_{i_m} f_\eps(\lambda x)x_{i_1}\dots x_{i_m}\psi_\eps(x)\,dx}\\
\le \sum_{i_1,\dots,i_m=1}^d\sup_{y\in K\cdot\supp(\psi)}\abs{\partial_{i_1}\dots\partial_{i_m} f_\eps(y)}\cdot\sup_{x\in \supp(\psi)}\abs{x_{i_1}\dots x_{i_m}\psi_\eps(x)},
\end{multline*}
$(F_\eps(\lambda))_\eps\in{\mathcal E}_M(\R^+)$ (and the definition is independent of representatives). Hence $F\in\Gen(\R^+)$, where $F(\lambda)=\int_{\R^d}f(\lambda x)\psi(x)\,dx$.\\
Further, for $\mu\in\R^+$ and $\lambda\in\GenR^+_c$,
\[
F(\mu\lambda)=\int_{\R^d}f(\mu \lambda x)\psi(x)\,dx
=\frac{1}{\lambda^d}\int_{\R^d} f(\mu y)\psi\Big(\frac{y}{\lambda}\Big)\,dy.
\]
Since $1/\lambda\in\GenR^+_c$, the hypotheses imply that
\[
F(\mu\lambda)
=\frac{\mu^{\alpha}}{\lambda^d}\int_{\R^d} f(y)\psi\Big(\frac{y}{\lambda}\Big)\,dy
=\mu^\alpha \int_{\R^d}f(\lambda x)\psi(x)\,dx=\mu^\alpha F(\lambda).
\]
\end{proof}
\begin{cor}
Let $f\in\Gen(\Omega)$ be weakly homogeneous of degree $\alpha$. Then $f$ satisfies equation (\ref{eqn_weak_homog}) for each $\lambda\in\GenR_c^+$ and each $\phi\in\Gen_c^\infty(\Omega)$.
\end{cor}
\begin{proof}
Let $\phi\in\Gen_c^\infty(\Omega)$. Let $F(\lambda)=\int_{\R^d} f(\lambda x)\phi(x)\,dx$. As $F\in\Gen(\R^+)$ is strongly homogeneous of degree $\alpha$, by \cite[Lemma~4.3]{homog}, $F(\lambda) = c\lambda^\alpha$, for some $c\in\GenC$. Since $c=F(1)$, the point values at $\lambda\in\GenR_c^+$ yield the result.
\end{proof}

\begin{thm}\label{thm_weak_Euler}
Let $f\in\Gen(\Omega)$. Then $f$ is weakly homogeneous of degree $\alpha$ iff $f$ satisfies the corresponding Euler equation in the sense of generalized distributions, i.e.,
\[\int_{\R^d} \sum_{i=1}^d x_i \partial_i f(x) \phi(x)\,dx=\alpha\int_{\R^d} f(x)\phi(x)\,dx,\]
for each $\phi\in\Gen_c^\infty(\Omega)$.
\end{thm}
\begin{proof}
$\Rightarrow$: Let $\phi\in\Gen_c^\infty(\Omega)$. Let $F(\lambda)=\int_{\R^d} f(\lambda x)\phi(x)\,dx\in\Gen(\R^+)$. As $F(\lambda) = c\lambda^\alpha$, for some $c\in\GenC$,
\[0=\frac{d}{d\lambda}\Big(\frac{F(\lambda)}{\lambda^\alpha}\Big) = \frac{\lambda F'(\lambda)-\alpha F(\lambda)}{\lambda^{\alpha + 1}}\]
as a generalized function in $\Gen(\R^+)$. Taking the point value at $\lambda = 1$, $F'(1) = \alpha F(1)$. The result follows by differentiation under the integral sign.\\
$\Leftarrow$: Let $\phi\in\Gen_c^\infty(\Omega)$. Let $F\in\Gen(\R^+)$ be as before. Let $G(\lambda)=\frac{F(\lambda)}{\lambda^\alpha}$. Then
\begin{align*}
G'(\lambda)&= \frac{1}{\lambda^{\alpha+1}}\int_{\R^d}\Big(\sum_{i=1}^d \partial_i f(\lambda x) \lambda x_i - \alpha f(\lambda x)\Big)\phi(x)\,dx\\
&=\frac{1}{\lambda^{\alpha+d+1}}\int_{\R^d}\Big(\sum_{i=1}^d \partial_i f(x) x_i - \alpha f(x)\Big)\phi\Big(\frac{x}{\lambda}\Big)\,dx= 0
\end{align*}
by the hypotheses, since $x\mapsto\phi(x/\lambda)\in\Gen_c^\infty(\Omega)$, for each $\lambda\in\GenR^+_c$. Hence $G$ is a generalized constant \cite[1.2.35]{GKOS}, i.e., $G(\lambda)=G(1)$, $\forall \lambda\in\GenR^+_c$.
\end{proof}

\begin{thm}\label{thm_weak_homog_pierced}
Let $f\in\Gen(\R\setminus\{0\})$ be weakly homogeneous of degree $\alpha$. Then there exist $c_1,c_2\in\GenC$ such that for each $\phi\in\test(\R\setminus\{0\})$,
\[\int_\R f(x)\phi(x)\,dx = c_1\int_{\R^-} \abs{x}^\alpha\phi(x)\,dx + c_2\int_{\R^+} \abs{x}^\alpha\phi(x)\,dx.\]
\end{thm}
\begin{proof}
Let $g(x)=\frac{f(x)}{\abs{x}^\alpha}$. For $\phi\in\test(\R\setminus\{0\})$ and $\lambda\in\R^+$, since $\frac{\phi(x)}{\abs{x}^\alpha}\in\test(\R\setminus\{0\})$,
\[
\int_\R g(\lambda x) \phi(x)\, dx
= \int_\R f(\lambda x) \frac{\phi(x)}{\lambda^\alpha \abs{x}^\alpha}\,dx
= \int_\R f(x) \frac{\phi(x)}{\abs{x}^\alpha}\,dx
= \int_\R g(x) \phi(x)\,dx.
\]
Since $g\in\Gen(\R\setminus\{0\})$, by theorem \ref{thm_weak_Euler},
\[\int_\R g'(x)x\phi(x)\,dx =0, \quad\forall\phi\in\test(\R\setminus\{0\}).\]
Since $\test(\R\setminus\{0\})=\{x\phi:\phi\in\test(\R\setminus\{0\})\}$ and by partial integration,
\[\int_\R g(x)\phi'(x)\,dx=0, \quad\forall\phi\in\test(\R\setminus\{0\}).\]
Fix $\phi_1\in\test(\R\setminus\{0\})$ with $\int_{\R^-}\phi_1=1$, $\int_{\R^+}\phi_1=0$ and $\phi_2\in\test(\R\setminus\{0\})$ with $\int_{\R^-}\phi_2=0$, $\int_{\R^+}\phi_2=1$. Let $\phi\in\test(\R\setminus\{0\})$ arbitrary. Call $\psi=\phi - (\int_{\R^-}\phi)\phi_1 - (\int_{\R^+}\phi)\phi_2$.
Since $\big\{\phi':\phi\in\test(\R\setminus\{0\})\big\} = \big\{\phi\in\test(\R\setminus\{0\}): \int_{\R^-}\phi = \int_{\R^+}\phi =0\big\}$,
$\int_\R g(x)\psi(x)\,dx=0$, hence with $c_j=\int_\R g(x)\phi_j(x)\,dx$ ($j=1,2$),
\[\int_\R g(x)\phi(x)\,dx = c_1\int_{\R^-}\phi(x)\,dx + c_2\int_{\R^+}\phi(x)\,dx,\]
for any $\phi\in\test(\R\setminus\{0\})$. The statement follows from the fact that $\test(\R\setminus\{0\}) = \{\frac{\phi}{\abs{x}^\alpha}:\phi\in\test(\R\setminus\{0\})\}$.
\end{proof}

The following example shows that the extension of the previous result to the case $\Omega=\R$ is nontrivial.
\begin{example}[Weak point support]\label{ex_weak_point_support}
There exists $f\in\Gen(\R)$ such that $\int_\R f\phi=0$, $\forall\phi\in\test(\R\setminus\{0\})$, but yet
$\int_\R f\phi_0\ne 0$ for some $\phi_0\in\test(\R)$ with $D^k\phi_0(0)=0$, $\forall k\in\N$.\\
In particular, for $\phi\in\test(\R)$, $\int_\R f\phi$ is not determined by $D^k\phi(0)$, $k\in\N$, and in particular, $f$ is not weakly equal to a linear combination of derivatives of $\delta$.
\end{example}
\begin{proof}
Let $\caninf=[(\eps)_\eps]\in\GenR$ and let $f(x)=\caninf^{x\abs{\ln\caninf}}$, for $x\in\GenR_c$, $x\ge\caninf$. Since for each $k\in\N$, $\sup_{x\ge 0} \abs{D^k(\eps^{x\abs{\ln\eps}})} = \sup_{x\ge 0}\abs{\ln\eps}^{2k}\eps^{x\abs{\ln\eps}} \le\abs{\ln\eps}^{2k} \le\eps^{-1}$, for small $\eps$, $f$ can be extended by means of a moderate cut-off to a function in $\Gen(\R)$ with $f\ge 0$ and $\restr{f}{\R^-}=0$. As $f(x)=0$, for each $x\in\GenR^+_c$, $\int_\R f\phi= 0$, for each $\phi\in\test(\R\setminus\{0\})$. Let $\phi_0(x)=e^{\frac{1}{x(x-2)}}\in\test([0,2])$.
Let $\rho=1/\abs{\ln\caninf}$. Since $f\phi_0\ge 0$ and $\alpha\le\rho\le 1$,
\[
\int_\R f\phi_0
\ge \int_{\rho}^{2\rho}\caninf^{x\abs{\ln\caninf}} e^{\frac{1}{x(x-2)}}\,dx \ge \caninf^2 e^{\frac{1}{\rho(\rho-2)}}\rho,
\]
since $\caninf^{x\abs{\ln\caninf}}$ is decreasing and $e^{\frac{1}{x(x-2)}}$ is increasing on $[0,1]$. As $e^{\frac{1}{\rho(\rho-2)}}\ge e^{-\frac{1}{\rho}}=\caninf$, $\int_\R f\phi_0\ne 0$.
\end{proof}

\begin{lemma}\label{lemma_princ_val}
Let $m\in\N$, $m\ge 1$. Let $\phi_0\in\test(\R)$ with $\phi_0=1$ on a neighbourhood of $0$.
\begin{enumerate}
\item There exist constants $c_j\in\C$ such that for each $\phi\in\test(\R)$,
\[\int_\R \Big(\phi(x) - \sum_{j=0}^{m-1}\frac{D^{j}\phi(0)}{j!}x^j\phi_0(x)\Big)\,\frac{dx}{x^m} = \pairing{x^{-m}}{\phi} + \sum_{j=0}^{m-1}c_j D^{j}\phi(0).\]
\item There exist constants $c_j\in\C$ such that for each $\phi\in\test(\R)$,
\[\int_{\R^+} \Big(\phi(x) - \sum_{j=0}^{m-1}\frac{D^{j}\phi(0)}{j!}x^j\phi_{0}(x)\Big)\,\frac{dx}{x^m} = \pairing{\FinPart{\frac{H(x)}{x^m}}}{\phi} + \sum_{j=0}^{m-1}c_j D^{j}\phi(0).\]
\end{enumerate}
\end{lemma}
\begin{proof}
(1) Follows easily using the fact that
\[
\pairing{x^{-m}}{\phi}=\lim_{R\to+\infty} \int_{-R}^R \Big(\phi(x) - \sum_{j=0}^{m-1} \frac{D^j\phi(0)}{j!} x^j\Big)\,\frac{dx}{x^m}.\]
(The last equation holds, e.g., by induction using the formula $D x^{-m} = -m x^{-m-1}$ \cite[\S 2.4]{Estrada} and partial integration.)\\
(2) Follows easily using the fact that \cite[\S 2.4]{Estrada}
\begin{multline*}
\pairing{\FinPart{\frac{H(x)}{x^m}}}{\phi}\\
=\int_1^\infty \frac{\phi(x)}{x^m}\,dx + \int_0^1 \Big(\phi(x) - \sum_{j=0}^{m-1} \frac{D^j\phi(0)}{j!} x^j\Big)\,\frac{dx}{x^m} - \sum_{j=0}^{m-2}\frac{D^j\phi(0)}{j!(m-j-1)}.
\end{multline*}
\end{proof}

\begin{prop}\label{prop_weakly_0_outside_origin}
Let $\caninf=[(\eps)_\eps]\in\GenR$. Let $f\in\Gen(\R)$ with $\int_\R f\phi=0$, $\forall \phi\in\test(\R)$ with $D^j\phi(0)=0$, $\forall j\in\N$. Let $K\csubset\R$ and $n\in\N$. Then
\[
(\exists m\in\N) (\forall\phi\in\test(K) \text{ with } D^j\phi(0)=0,\forall j\le m) \Big(\abs[\Big]{\int_\R f\phi} \le \caninf^n\Big).
\]
\end{prop}
\begin{proof}
Let $K\csubset\R$ and let $R\in\R$ such that $\abs{x}\le R/2$, $\forall x\in K$. Let $L=\{x\in\R: \abs x\le R\}\csubset\R$. Let $n\in\N$. Let $E=\{\phi\in\test(L): D^j\phi(0)=0, \forall j\in\N\}$. $E$ is a Fr{\'e}chet space, as a closed subspace of the Fr{\'e}chet space $\test(L)$. In particular, we can apply theorem \ref{thm_asympt_uniform_bd}(3) to the space $E$ and the net $(\phi\mapsto \eps^{-n-1}\int_\R f_\eps \phi)_\eps\in E'^{(0,1]}$, which is asymptotically bounded in $E'$ (with the topology of pointwise convergence). Hence we find $M\in\N$ for which
\[(\exists\eps_0\in (0,1]) (\forall\eps\le\eps_0) (\forall \phi \in E)
\big(\abs[\Big]{\int_\R f_\eps \phi}\le M \eps^{n+1} \max_{k\le M} \sup_{x\in\R}\abs{D^k \phi(x)}\big).\]
Now let $\Psi\in\test(L)$ with $\Psi(x)=1$, for $\abs{x}\le R/2$. Let $\chi_{\R\setminus [-\eps,\eps]}$ be the characteristic function of $\R\setminus[-\eps,\eps]$ and let $\rho_\eps(x)=\eps^{-1}\rho(x/\eps)$, where $\rho\in\test([-1/2,1/2])$ with $\int_\R\rho=1$. Let $\psi_\eps=\Psi \cdot (\chi_{\R\setminus [-\eps,\eps]}\conv \rho_\eps)\in E$. Then the net $(\psi_\eps)_\eps$ represents $\psi\in\Gen(\R)$ with $\psi(\tilde x)=0$ for $\abs{\tilde x}\le \caninf/2$ and $\psi(\tilde x)=1$ for $2\caninf\le\abs{\tilde x}\le R/2$. Further, $\sup_{x\in\GenR} \abs{D^j\psi(x)}\le c_j\alpha^{-j}$, for some $c_j\in\R$. Since $f\in\Gen(\R)$, there exists $N\in\N$ such that $\sup_{\abs x\le R}\abs{f(x)}\le\caninf^{-N}$. Now let $\phi\in\test(K)$ with $D^j\phi(0)=0$, $\forall j\le m:=n + 1 + \max(M,N)$. Then by the Taylor expansion, there exists $C\in\R$ such that $\abs{D^j\phi(x)}\le C\abs{x}^{m-j}$, for each $x\in\R$ and $j\le m$. Further, also $\phi\psi_\eps\in E$, $\forall\eps$, so
\[
\abs[\Big]{\int_\R f \phi\psi}\le M \caninf^{n+1} \max_{k\le M} \sup_{x\in\GenR}\abs{D^k (\phi\psi)(x)}.
\]
Now for $k\le M$,
\begin{multline*}
\sup_{x\in\GenR}\abs{D^k (\phi(1-\psi))(x)}
\le 2^k \sup_{\abs{x}\le 2\alpha, j\le k}\abs{D^j\phi(x)D^{k-j}(1-\psi)(x)}\\
\le C' \max_{j\le k}\alpha^{m-j}\alpha^{-(k-j)}\le C',
\end{multline*}
for some $C'\in\R$. Further,
\[
\abs{\int_\R f \phi(1-\psi)}=\abs{\int_{-2\caninf}^{2\caninf} f \phi(1-\psi)}\le \caninf^{-N}\caninf^m\le \alpha^{n+1}.
\]
Hence
\[\abs{\int_\R f \phi}\le \alpha^{n+1} + M \caninf^{n+1} \max_{k\le M} \sup_{x\in\R}\abs{D^k \phi(x)} + M \caninf^{n+1} C'\le\alpha^n.\]
\end{proof}

\begin{thm}\label{thm_weak_hom_char_dim1}\leavevmode
\begin{enumerate}
\item Let $f\in\Gen(\R)$ be weakly homogeneous of degree $-m$, $m=1,2,3,\dots$. Then there exist $c_1,c_2\in\GenC$ such that for each $\phi\in\test(\R)$,
\[\int_\R f(x)\phi(x)\,dx = c_1 D^{m-1} \phi(0) + c_{2} \pairing{x^{-m}}{\phi}.\]
\item Let $f\in\Gen(\R)$ be weakly homogeneous of degree $\alpha \in\R$, $\alpha\ne -1,-2,-3,\dots$. Then there exist $c_1,c_2\in\GenC$ such that for each $\phi\in\test(\R)$,
\[\int_\R f(x)\phi(x)\,dx = c_{1}\pairing{x^\alpha_-}{\phi} + c_{2}\pairing{x^\alpha_+}{\phi},\]
where $x^\alpha_-, x^\alpha_+\in\test'(\R)$ are as in \cite[\S 2.4]{Estrada}. 
\end{enumerate}
\end{thm}
\begin{proof}
(1) Let $g(x)=x^mf(x)\in\Gen(\R)$. As in theorem \ref{thm_weak_homog_pierced},
\[\int_\R g'(x)x\phi(x)\,dx=0, \quad\forall\phi\in\test(\R).\]
Since $\{x\phi(x): \phi\in\test(\R)\}=\{\phi\in\test(\R): \phi(0)=0\}$ and by partial integration,
\[\int_\R g(x)\phi'(x)\,dx=0, \quad\forall\phi\in\test(\R)\text{ with }\phi(0)=0.\]
Since $\{\phi':\phi\in\test(\R)$ with $\phi(0)=0\}=\{\phi\in\test(\R): \int_{\R}\phi=\int_{\R^+}\phi=0\}$, we find as in theorem \ref{thm_weak_homog_pierced} $c_1,c_2\in\GenC$ for which
\[\int_\R f(x)x^{m}\phi(x)\,dx=c_1\int_{\R}\phi(x)\,dx + c_2\int_{\R^+}\phi(x)\,dx, \forall\phi\in\test(\R).\]
Let $\phi_0$ as in lemma \ref{lemma_princ_val}. Since for any $\phi\in\test(\R)$, $\psi_\phi:=\phi - \sum_{j=0}^{m-1}\frac{D^{j}\phi(0)}{j!}x^{j}\phi_0\in \{x^{m}\phi(x): \phi\in\test(\R)\}=\{\phi\in\test(\R): D^j\phi(0)=0, 0\le j < m\}$, there exist $a_j\in\GenC$ such that for any $\phi\in\test(\R)$,
\[\int_\R f(x)\phi(x)\,dx = \sum_{j=0}^{m-1} a_{j}D^{j}\phi(0) + c_1\int_\R\frac{\psi_\phi(x)}{x^m}\,dx + c_2\int_{\R^+}\frac{\psi_\phi(x)}{x^m}\,dx.\]
Hence by lemma \ref{lemma_princ_val}, there exist $a_j\in\GenC$ such that for any $\phi\in\test(\R)$,
\[
\int_\R f(x)\phi(x)\,dx = \sum_{j=0}^{m-1} a_{j}D^{j}\phi(0) + c_1\pairing{x^{-m}}{\phi} + c_2\pairing{\FinPart{\frac{H(x)}{x^m}}}{\phi}.
\]
Expressing the homogeneity of order $-m$ in this equation, we find that $c_2=0$ and $a_j=0$ for $j<m-1$, 
since \cite[\S 2.6]{Estrada} $x^{-m}$ is homogeneous of degree $-m$ and there exists $c\in\R\setminus\{0\}$ such that
\[\lambda^m\FinPart{\frac{H(\lambda x)}{(\lambda x)^m}}=\FinPart{\frac{H(x)}{x^m}}+ c\ln\lambda \,D^{m-1}\delta,\quad\forall\lambda\in\R^+.\]
(2) Case~1: $\alpha=m\in\N$. Notice that $\frac{f(x)}{x^m}$ is not necessarily in $\Gen(\R)$. Let $\phi\in\test(\R)$ with $D^j \phi(0)=0$ for $0\le j \le m$. Then $\frac{\phi(x)}{x^m}\in\test(\R)$, 
$\frac{\phi'(x)}{x^m}\in\test(\R)$ and $\frac{\phi(x)}{x^{m+1}}\in\test(\R)$, hence
\[0= \int_\R \left(f(x)\frac{\phi(x)}{x^m}\right)'dx =\int_\R f'(x)\frac{\phi(x)}{x^m}\,dx + \int_\R f(x)\frac{\phi'(x)}{x^m}\,dx - m \int_\R f(x)\frac{\phi(x)}{x^{m+1}}\,dx.\]
By theorem~\ref{thm_weak_Euler}, $\int_\R xf'(x)\frac{\phi(x)}{x^{m+1}}\,dx= m \int_\R f(x)\frac{\phi(x)}{x^{m+1}}\,dx$. Hence
\[\int_\R f(x)\frac{\phi'(x)}{x^m}\,dx=0,\quad\forall\phi\in\test(\R) \text{ with }D^j \phi(0)=0, 0\le j\le m.\]
Since $\{\phi':\phi\in\test(\R)$ with $D^j \phi(0)=0$, for $0\le j\le m\}=\{\phi\in\test(\R): \int_{\R}\phi=\int_{\R^+}\phi=\phi(0)=\cdots= D^{m-1} \phi(0)=0\}=\{\phi\in\test(\R): \int_{\R}\phi=\int_{\R^+}\phi=0$ and $\frac{\phi(x)}{x^m}\in\test(\R)\}$,
\[
\int_\R f(x)\phi(x)\,dx=0,\quad\forall\phi\in\test(\R)\text{ with }\int_\R x^m\phi(x)\,dx=\int_{\R^+}x^m\phi(x)\,dx=0.\]
As before, we find $c_1,c_2\in\GenC$ for which
\[
\int_\R f(x)\phi(x)\,dx= c_1\int_{\R}x^m\phi(x)\,dx + c_2\int_{\R^+}x^m\phi(x)\,dx,\quad\forall\phi\in\test(\R).
\]
Case~2: $\alpha\notin\Z$. If $\phi\in\test(\R)$ with $D^j\phi(0)$, $\forall j\in\N$, then $\frac{\phi(x)}{\abs{x}^\alpha}\in\test(\R)$, $\frac{\phi'(x)}{\abs{x}^\alpha}\in\test(\R)$ and $\frac{\phi(x)}{\abs{x}^{\alpha+1}}\in\test(\R)$. As before, since $(\abs x^{-\alpha})'=-\alpha x^{-1}\abs{x}^{-\alpha}$ for $x\ne 0$,
\[\int_\R f(x)\frac{\phi'(x)}{\abs{x}^\alpha}\,dx=0,\quad\forall\phi\in\test(\R) \text{ with }D^j \phi(0)=0, \forall j\in\N.\]
Let $\psi\in\test(\R)$ with $D^j \psi(0)=0, \forall j\in\N$ and $\int_{\R^-} \abs{x}^\alpha\psi=\int_{\R^+}\abs{x}^\alpha \psi=0$. Then $\phi(x):=\int_{-\infty}^x\abs{t}^\alpha\psi(t)\,dt\in\test(\R)$,
$D^j \phi(0)=0, \forall j\in\N$ and $\psi(x)=\phi'(x)/\abs{x}^\alpha$. Hence $\int_\R f\psi=0$, for such $\psi$.\\
As before, fixing $\phi_1\in\test(\R)$ with $D^j\phi_1(0)=0$, $\forall j\in\N$, $\int_{\R^-}\abs{x}^\alpha\phi_1=1$, $\int_{\R^+}\abs{x}^\alpha\phi_1=0$ and $\phi_2\in\test(\R)$ with $D^j\phi_2(0)=0$, $\forall j\in\N$, $\int_{\R-}\abs{x}^\alpha\phi_2=0$, $\int_{\R^+}\abs{x}^\alpha\phi_2=1$, we find $c_1$, $c_2\in\GenC$ such that for each $\phi\in\test(\R)\text{ with }D^j \phi(0)=0, \forall j\in\N$,
\[
\int_\R f \phi = c_1\int_{\R^-} \abs{x}^\alpha\phi(x)\,dx + c_2\int_{\R^+}\abs{x}^\alpha \phi(x)\,dx=c_1\pairing{x^\alpha_-}{\phi}+ c_2\pairing{x^\alpha_+}{\phi}.
\]
Let $\iota$ be an embedding of $\test'(\R)$ into $\Gen(\R)$ (which preserves the pairing). For $g=f-c_1\iota(x^\alpha_-)-c_2\iota(x^\alpha_+)$, we obtain $\int_\R g\phi=0$, $\forall \phi\in\test(\R)\text{ with }D^j \phi(0)=0, \forall j\in\N$. Further, since $x^\alpha_-$, $x^\alpha_+$ are homogeneous distributions of degree $\alpha$ \cite[\S 2.6]{Estrada}, also $g$ is weakly homogeneous of degree $\alpha$ \cite[Prop.~4.21]{homog}.\\
Now let $m\in\N$ and $\phi\in\test(\R)$ with $D^j\phi(0)=0$, for each $j\ne m$. Let $\lambda\in\R^+$. Then $D^j(\phi(x)-\lambda^m\phi(x/\lambda))(0)=0$, $\forall j\in\N$, so
\[
\int_\R g(x) \phi(x)\,dx = \int_\R g(x)\lambda^m\phi\left(\frac x \lambda\right)\,dx
=\lambda^{\alpha + m + 1}\int_\R g(x)\phi(x)\,dx.
\]
As $\alpha\ne -m -1$ and $\lambda\in\R^+$ arbitrary, $\int_\R g\phi = 0$.\\
Now let $\phi\in\test(\R)$ arbitrary. Let $\psi\in\test(\R)$ with $\psi(x)=1$, $\forall x\in\supp\phi$ and $\forall x$ in a neighbourhood of $0$. Then for each $m\in\N$, $\phi(x)=\sum_{j=0}^m \frac{D^j\phi(0)}{j!}x^j\psi(x) + x^m R_m(x)$,
where $R_m\in\test(\R)$ with $\supp R_m\subseteq \supp\psi$. Since $D^k\big(\frac{D^j\phi(0)}{j!}x^j\psi\big)(0)=0$, $\forall k\ne j$, we have $\int_\R g\phi=\int_\R g(x)x^m R_m(x)\,dx$. Let $n\in\N$. By proposition \ref{prop_weakly_0_outside_origin}, there exists $m\in\N$ such that $\abs{\int_\R g(x) x^m R_m(x)\,dx}\le\caninf^n$. Hence $\int_\R g\phi=0$.
\end{proof}

\section{Weak homogeneity in $\Gen(\R^d\setminus\{0\})$}
We show an analogon of the formula $f(x)= f\big(\frac{x}{\abs x}\big) \abs{x}^\alpha$, which holds for (strongly) homogeneous generalized functions of degree $\alpha$ on $\R^d\setminus\{0\}$ \cite{homog}. The following example shows that (similar to distribution theory), we cannot maintain the equation in unchanged form (even in the sense of generalized distributions).
\begin{example}
Let $\alpha\in\R$. There exists $f\in\Gen(\R^d\setminus\{0\})$ such that $\int_{\R^d}f\phi=0$, $\forall\phi\in\test(\R^d\setminus\{0\})$, but yet $\int_{\R^d}f\big(\frac{x}{\abs x}\big) \abs{x}^\alpha\phi(x)\,dx\ne 0$, for some $\phi\in\test(\R^d\setminus\{0\})$.
\end{example}
\begin{proof}
Let $\rho\in\Schwartz(\R)$ with $\rho(0)\ne 0$ and with $\int_\R t^m\rho(t)\,dt=0$, $\forall m\in\N$, $m\ge 1$ (for the existence of $\rho$, see e.g., \cite[\S 1.2.2]{GKOS}). Let $g_\eps(t)= \eps^{-1}\big[2\rho\big(\frac{2t-2}{\eps}\big) - \rho\big(\frac{t-1}{\eps}\big)\big]$. Then $g=[(g_\eps)_\eps]\in\Gen(\R)$. Let $f(x)=g(\abs{x})\in\Gen(\R^d\setminus\{0\})$.
For $\psi\in\test(\R^+\times S^{d-1})$,
\[
\sup_{\omega\in S^{d-1}}\abs{\int_\R g_\eps(r)\psi(r,\omega)\,dr} = \sup_{\omega\in S^{d-1}}\abs{\int_\R [2\rho(2r)-\rho(r)]\psi(1 + \eps r,\omega)\,dr}.
\]
Since for any $m\in\N$, $\psi(1 + \eps r,\omega)=\psi(1,\omega) + (\eps r)\partial_r \psi(1,\omega) +\cdots + \frac{(\eps r)^m}{m!}\partial_r^m \psi(1+\eps r', \omega)$, the moment conditions imply that
\[
\sup_{\omega\in S^{d-1}}\abs{\int_\R g_\eps(r)\psi(r,\omega)\,dr}
\le \frac{\eps^m}{m!}
\sup_{\omega\in S^{d-1},r\in\R} \abs{\partial_r^m\psi(r,\omega)}
\int_{\R}r^m(2\abs{\rho(2r)}+\abs{\rho(r)})\,dr.
\]
Now let $\phi\in\test(\R^d\setminus\{0\})$. With $\psi(r,\omega):=r^{d-1}\phi(r\omega)$, if $r\ge 0$, and $\psi(r,\omega)=0$, if $r\le 0$, we obtain $\psi\in\test(\R^+\times S^{d-1})$. Hence
\begin{multline*}
\abs{\int_{\R^d} f_\eps\phi}=\abs{\int_{S^{d-1}} d\omega \int_0^\infty r^{d-1} g_\eps(r)\phi(r\omega)\,dr}\\
\le \mathrm{vol}(S^{d-1})\sup_{\omega\in S^{d-1}}\abs{\int_\R g_\eps(r)\psi(r,\omega)\,dr}\le \eps^m,
\end{multline*}
for small $\eps$. As $m\in\N$ arbitrary, $\int_{\R^d} f\phi=0$, $\forall\phi\in\test(\R^d\setminus\{0\})$.
Yet
\[\int_{\R^d} f\left(\frac{x}{\abs{x}}\right)\abs{x}^\alpha\phi(x)\,dx = g(1)\int_{\R^d}\abs{x}^\alpha \phi(x)\,dx\ne 0\]
for some $\phi\in\test(\R^d\setminus\{0\})$, since $g(1)=[(\eps^{-1})_\eps]\rho(0)\ne 0$.
\end{proof}
Hence $\int_{\R^d} f\big(\frac{x}{\abs{x}})\abs{x}^\alpha\phi(x)\,dx$ is not stable when we add to $f$ a function $g$ with $\int_{\R^d}g\phi=0$, $\forall\phi\in\test(\R^d\setminus\{0\})$.  
Instead of the values $f(\omega)$ ($\omega\in S^{d-1}$), the property that $\frac{f(r\omega)}{r^\alpha}$ is constant in $r$ for a homogeneous function $f$ of degree $\alpha$, suggests that for a fixed $u_0\in\test(\R^+)$ with $\int_\R u_0=1$, $g(\omega):=\int_{\R^+} \frac{f(r\omega)}{r^\alpha}u_0(r)\,dr$ might be a better candidate. As in lemma \ref{lemma_F}, one easily sees that $g\in\Gen(\R^d\setminus\{0\})$, if $f\in\Gen(\R^d\setminus\{0\})$.
We first show that the choice of $u_0$ is irrelevant.

\begin{lemma}\label{lemma_weak_Euler_polar_coord}
If $f\in\Gen(\R^d\setminus\{0\})$ is weakly homogeneous of degree $\alpha$, then
\[\int_{\R^d} \frac{f(x)}{\abs{x}^{\alpha+d-1}}u'(\abs x)v\left(\frac{x}{\abs x}\right)\,dx=0, \quad\forall u\in\Gen_c^\infty(\R^+),\quad \forall v\in\Gen_c^\infty(\R^d).\]
\end{lemma}
\begin{proof}
Let $u\in\Gen_c^\infty(\R^+)$. By partial integration,
\[
\int_{\R^+} \frac{f(r\omega)}{r^\alpha}u'(r)\,dr
=\int_{\R^+}\left(\alpha \frac{f(r\omega)}{r^{\alpha+1}} - \frac{\nabla f(r\omega)\cdot \omega}{r^\alpha}\right)u(r)\,dr.
\]
Hence by theorem \ref{thm_weak_Euler}, for $u\in\Gen_c^\infty(\R^+)$ and $v\in\Gen_c^\infty(\R^d)$,
\begin{multline*}
\int_{\R^d} \frac{f(x)}{\abs{x}^{\alpha+d-1}}u'(\abs x)v\left(\frac{x}{\abs x}\right)\,dx
= \int_{S^{d-1}} v(\omega) \Big(\int_{\R^+} \frac{f(r\omega)}{r^\alpha}u'(r)\,dr\Big)\,d\omega\\
= \int_{\R^d}\big(\alpha f(x) - \nabla f(x)\cdot x\big) \underbrace{\frac{u(\abs x)}{\abs{x}^{\alpha + d}}\, v\left(\frac{x}{\abs x}\right)}_{\in\Gen_c^\infty(\R^d\setminus\{0\})}\,dx =0.
\end{multline*}
\end{proof}

\begin{prop}\label{prop_indep_of_u_0}
Let $f\in\Gen(\R^d\setminus\{0\})$ be weakly homogeneous of degree $\alpha$. Let $u_1, u_2\in\test(\R^+)$ with $\int_{\R} u_1=\int_{\R} u_2=1$. Let $g_j(\omega)=\int_{\R^+} \frac{f(r\omega)}{r^\alpha}u_j(r)\,dr$, for $\omega\in S^{d-1}$ ($j=1,2$). Then
\[\int_{\R^d} g_1\left(\frac{x}{\abs{x}}\right)\abs{x}^\alpha \phi(x)\,dx
=\int_{\R^d} g_2\left(\frac{x}{\abs{x}}\right)\abs{x}^\alpha \phi(x)\,dx, \quad \forall\phi\in\test(\R^{d}\setminus\{0\}).\]
\end{prop}
\begin{proof}
For $\psi\in\Gen_c^\infty(\R^d)$,
\begin{multline*}
\int_{S^{d-1}}(g_2(\omega)-g_1(\omega))\psi(\omega)\,d\omega
=\int_{S^{d-1}}\int_{\R^+} \frac{f(r\omega)}{r^\alpha}(u_2(r)-u_1(r))\psi(\omega)\, d\omega dr\\
=\int_{\R^d} \frac{f(x)}{\abs{x}^{\alpha+d-1}}\big(u_2(\abs x)-u_1(\abs x)\big)\psi\left(\frac{x}{\abs x}\right)\,dx=0,
\end{multline*}
by lemma \ref{lemma_weak_Euler_polar_coord} and by the fact that $\{u': u\in\test(\R^+)\}=\{u\in\test(\R^+): \int_{\R}u=0\}$.\\
Thus, for $\phi\in\test(\R^{d}\setminus\{0\})$, there exists $N\in\N$ such that
\begin{multline*}
\int_{\R^d} \left(g_2\left(\frac{x}{\abs{x}}\right) - g_1\left(\frac{x}{\abs{x}}\right)\right) \abs{x}^\alpha \phi(x)\,dx\\
= \int_{\frac 1 N}^N\left(\int_{S^{d-1}} \big(g_2(\omega)-g_1(\omega)) r^{\alpha + d -1}\phi(r\omega) \,d\omega\right)dr = 0,
\end{multline*}
since $\psi(x):=r^{\alpha + d - 1}\phi(rx)\in\Gen_c^\infty(\R^d)$, $\forall r\in\GenR_c^+$.
\end{proof}

\begin{thm}\label{thm_weak_hom_char_dimd}
Let $f\in\Gen(\R^d\setminus\{0\})$ be weakly homogeneous of degree $\alpha$. Let $u_0\in\test(\R^+)$ with $\int_{\R} u_0=1$. Let $g(\omega)=\int_{\R^+} \frac{f(r\omega)}{r^\alpha}u_0(r)\,dr$, for $\omega\in S^{d-1}$. Then
\[\int_{\R^d} f(x)\phi(x)\,dx = \int_{\R^d} g\left(\frac{x}{\abs{x}}\right)\abs{x}^\alpha \phi(x)\,dx,\quad\forall\phi\in\test(\R^d\setminus\{0\}).
\]
In particular, $f$ is equal in the sense of generalized distributions to a (strongly) homogeneous generalized function of degree $\alpha$ in $\Gen(\R^d\setminus\{0\})$.
\end{thm}
\begin{proof}
Let $u\in\test(\R^+)$ and $v\in\test(\R^d)$ and let $\phi(x)=u(\abs x)v\big(\frac{x}{\abs x}\big)$. Then
\begin{multline*}
\int_{\R^d} g\left(\frac{x}{\abs{x}}\right)\abs{x}^\alpha \phi(x) \,dx
= \int_{S^{d-1}}g(\omega) v(\omega)\,d\omega \int_{\R^+} u(r) r^{\alpha + d - 1}\, dr\\
= \int_{\R^d} \frac{f(x)}{\abs{x}^{\alpha+d-1}}u_0(\abs x)v\left(\frac{x}{\abs x}\right)\,dx \int_{\R^+} u(r) r^{\alpha + d - 1}\, dr,
\end{multline*}
where, by proposition \ref{prop_indep_of_u_0}, $u_0\in\test(\R^+)$ with $\int_\R u_0=1$ arbitrary. Suppose now that $\int_\R u(r) r^{\alpha + d - 1} \,dr\ne 0$. Then we can choose \[u_0(r) = \big(\int_\R u(r) r^{\alpha + d - 1} \,dr\big)^{-1} u(r) r^{\alpha + d - 1},\] and $\int_{\R^d} g\big(\frac{x}{\abs{x}}\big)\abs{x}^\alpha \phi(x) \,dx= \int_{\R^d} f\phi$ as required.\\
If $\int_\R u(r) r^{\alpha + d - 1} \,dr = 0$, then $w(r)= u(r) r^{\alpha + d-1}\in\{u': u\in\test(\R^+)\}$, hence by lemma \ref{lemma_weak_Euler_polar_coord}, $\int_{\R^d}f\phi=0$.
So also in this case, equality holds. For arbitrary $\phi\in\test(\R^d\setminus\{0\})$, the result follows by proposition \ref{prop_weak_equal_polar_test_functions}, since $g\big(\frac{x}{\abs x}\big)\abs{x}^\alpha\in\Gen(\R^d\setminus\{0\})$.
\end{proof}

We recall \cite[1.2.24]{GKOS} that
\begin{multline*}
{\mathcal E}_{\tau}(\R^d)=\big\{(u_\eps)_\eps\in\Cnt[\infty](\R^d)^{(0,1]}: (\forall\beta\in\N^d)(\exists N\in\N)\\
\big(\sup_{x\in\R^d}\abs[\big]{\partial^\beta u_\eps(x)}(1 + \abs x)^{-N}\le\eps^{-N},\text{ for small }\eps\big)\big\}.
\end{multline*}

The following corollary is somewhat similar to \cite[Conj.\ 4.24]{homog}.
\begin{cor}
Let $f\in\Gen(\R^d)$ such that $\restr{f}{\R^d\setminus\{0\}}$ is weakly homogeneous of degree $\alpha$. Then $f$ is equal in the sense of generalized distributions to some $h\in\Gen(\R^d)$ that admits a representative $(h_\eps)_\eps\in{\mathcal E}_{\tau}(\R^d)$.
\end{cor}
\begin{proof}
Let $f=[(f_\eps)_\eps]$ and let $g_\eps(x)=\int_{\R^+}\frac{f_\eps(rx)}{r^\alpha}u_0(r)\,dr$, with $u_0\in\test(\R^+)$, $\int_\R u_0=1$. Let $\chi\in\test(\R^d)$ with $\chi(x)=1$, if $\abs{x}\le 1$. Let $h_\eps(x)=f_\eps(x)\chi(x) + g_\eps\big(\frac{x}{\abs x}\big)\abs{x}^\alpha(1-\chi(x))$. It is easily checked that $(h_\eps)_\eps\in{\mathcal E}_{\tau}(\R^d)$. By the previous theorem, for each $\phi\in\test(\R^d)$,
\[\int_{\R^d} h\phi = \int_{\R^d}f\chi\phi + \int_{\R^d} g\left(\frac{x}{\abs x}\right)\abs{x}^\alpha \underbrace{(1-\chi(x))\phi(x)}_{\in\test(\R^d\setminus\{0\})}\,dx= \int_{\R^d} f\phi.\]
\end{proof}

The following example shows that under the assumptions of the previous corollary, $f$ need not have a representative in ${\mathcal E}_\tau(\R^d)$.
\begin{example}
There exists $f\in\Gen(\R)$ which is equal to $0$ in the sense of generalized distributions, which does not have a representative $(f_\eps)_\eps\in{\mathcal E}_\tau(\R)$.
\end{example}
\begin{proof}
Let $f\in\Gen(\R)$ with representative $(f_\eps(x))_\eps= (e^{ix/\eps}\eps^{-x} e^{x^2})_\eps$. Then for each $\phi\in\test(\R)$, there exists $N\in\N$ such that $\supp\phi\subseteq [-N,N]$, hence
\begin{multline*}
\abs[\bigg]{\integr_{\R} f_\eps\phi} =\abs[\Bigg]{\integr_{-N}^N e^{\frac{i-\eps\ln\eps}{\eps} x}e^{x^2}\phi(x)\,dx}
= \frac{\eps^m}{\abs{i-\eps\ln\eps}^m} \abs[\Bigg]{\integr_{-N}^N e^{\frac{i-\eps\ln\eps}{\eps} x} D^m(e^{x^2}\phi(x))\,dx}\\
\le \eps^m \eps^{-N} \sup_{x\in\R}\abs{D^m(e^{x^2}\phi(x))} = O(\eps^{m-N}),
\end{multline*}
for each $m\in\N$. Hence $f=0$ in the sense of generalized distributions. Further, for each $N\in\N$ and $\eps\in (0,1]$,
\[\sup_{\abs{x}\le N}\abs{f_\eps(x)}(1 + \abs{x})^{-N}\ge \abs{f_\eps(N)}(1 + N)^{-N}\ge \eps^{-N}\left(e^{N}(1+N)^{-1}\right)^{N}\ge\eps^{-N}.\]
Now let $(g_\eps)_\eps$ be an arbitrary representative of $f$. Then for each $N\in\N$, $\sup_{\abs{x}\le N}\abs{g_\eps(x)-f_\eps(x)}\le 1$ for small $\eps$. Hence for each $N\in\N$,
\[
\sup_{x\in\R} \frac{\abs{g_\eps(x)}}{(1 + \abs{x})^N}\ge
\sup_{\abs x\le N} \frac{\abs{f_\eps(x)}}{(1 + \abs{x})^{N}} - \sup_{\abs x\le N} \frac{\abs{g_\eps(x)-f_\eps(x)}}{(1 + \abs{x})^{N}}\ge \eps^{-N} - 1,
\]
for small $\eps$. Hence $(g_\eps)_\eps\notin{\mathcal E}_\tau(\R)$.
\end{proof}

\begin{rem}
The following notion of homogeneous generalized function $f=[(f_\eps)_\eps]\in\Gen(\R^d)$ was pointed out to us by V.~Shelkovich:
\begin{df}
For $\lambda\in\R^+$, let $H_\lambda(f)(x)=[(f_{\lambda\eps}(\lambda x))_\eps]\in\Gen(\R^d)$ (it is easy to see that this definition does not depend on the representative). Then we call $f$ \defstyle{regularized homogeneous} of degree $\alpha$ iff $H_\lambda(f)=\lambda^\alpha f$ in $\Gen(\R^d)$, $\forall\lambda\in\R^+$.
\end{df}
This notion also achieves good consistency with distributional homogeneity, if the embedding $\iota$: $\test'(\R^d)\to\Gen(\R^d)$ is realized as $\iota(u)=[(u\conv\rho_\eps)_\eps]$ with $\rho_\eps(x)=\eps^{-d} \rho(x/\eps)$, for some $\rho\in\Schwartz(\R^d)$ with the usual moment conditions \cite[\S 1.2.2]{GKOS}. It is easy to see that then $\iota(u)$ is regularized homogeneous of degree $\alpha$ iff $u$ is a homogeneous distribution of degree $\alpha$. Further, the product of two regularized homogeneous generalized functions of degree $\alpha$, resp.\ $\beta$, is regularized homogeneous of degree $\alpha+\beta$ (a property that also holds for strongly, but not for weakly homogeneous generalized functions). Therefore, it is an interesting notion that deserves to be explored. Nevertheless, we believe that the weak homogeneity described in \cite{homog} and in this paper has its value as an intrinsic notion on $\Gen(\R^d)$. For instance, constant generalized functions are not necessarily regularized homogeneous of degree $0$, and a generalized function $f$ representing a homogeneous distribution $u$ in the sense that $\int_{\R^d} f\phi= \pairing{u}{\phi}$, for each $\phi\in\test(\R^d)$, is not necessarily regularized homogeneous (consider $u=\delta$ and $f=[(\rho_{\eps^2})_\eps]$, where $\rho_\eps$ is defined as before).
\end{rem}

\end{document}